\documentclass[12pt, twoside, leqno]{article}

\usepackage{amsmath,amsthm}
\usepackage{amssymb}
\usepackage{enumerate}
\usepackage[T1]{fontenc}

\DeclareMathAlphabet{\mathpzc}{OT1}{pzc}{m}{it}

\makeindex

\newcommand{\CC}{\mathbb{C}}

\newcommand{\NN}{\mathbb{N}}

\newcommand{\RE}{ {\rm Re \,} }

\theoremstyle{plain}
\newtheorem{theorem}{Theorem}[section]

\newtheorem{lemma}[theorem]{Lemma}

\theoremstyle{definition}
\newtheorem{definition}[theorem]{Definition}

\theoremstyle{remark}
\newtheorem{remark}[theorem]{Remark}

\begin{document}


\title{\textit{The Weighted Bergman Kernel and the Green's Function}}

\author{by
\sc Steven Krantz \rm (Washington Univ. at St. Louis) \\
and
\sc Pawe{\l} M. W\'ojcicki \rm (Warsaw)
}

\date{}
\maketitle
\renewcommand{\thefootnote}{}
\renewcommand{\thefootnote}{\arabic{footnote}}
\setcounter{footnote}{0}

\maketitle

\begin{abstract}
We study the connection between weighted Bergman kernel and Green's function on a
domain $W\subset\CC$ for which the Green's function exists.
\end{abstract}

\section{Introduction}

The Bergman kernel (see for instance \cite{Bergman, Krantz1, Krantz2, Shabat, Skwarczynski; Mazur})
has become a very important tool in geometric function theory, both in one and several complex variables.
It turns out that not only the classical Bergman kernel, but also the weighted one can be useful
(see \cite{Englis, Englis1, Odzijewicz} for instance). Let $W\subset\CC$ be a domain, such that  the Bergman space $L^2_H(W)$ is a non-zero space and $G_W$ the Green's function of $W$ (let us recall that $G_W$ exists if $\CC\setminus W$ is not polar, and this is only if $L^2_H(W)\neq 0$---see \cite{Myrberg} and \cite{Carleson}).

It is known, in the classical case, that
$$
K_{W}(z,w)=-\frac{2}{\pi}\frac{\partial^2}{\partial z \partial \overline{w}}G_W(z,w)
$$
(see \cite{Suszczynski}) for $z, w\in W,\,z\neq w$ (it was originally proved in \cite{Bergman; Schiffer} with
additional assumptions on $\partial W$).
On the other hand, if $\partial W$ consists of a finite number of Jordan curves, $\rho(z)$ is a positive continuously differentiable function of $x$ and $y$ on a neighborhood of $\overline{W}$,
$K_{W,\,\rho}(z,w)$ a weighted Bergman kernel of the space $L^2_H(W,\,\rho)$
and $G_{W,\,\rho}$ the Green's function for an operator
$\displaystyle P_{\rho}=\frac{\partial}{\partial\overline{z}}\frac{1}{\rho(z)}\frac{\partial}{\partial z}$,
then
$$
K_{W,\,\rho}(z,w)=-\frac{2}{\pi\rho(z)\rho(w)}\frac{\partial^2}{\partial z \partial \overline{w}}G_{W,\,\rho}(z,w)
$$
(see \cite{Garabedian}). In this paper we prove that the connection above holds for any domain $W\subset\CC$, for which $L^2_H(W)\neq 0$. Along the way, we give a shorter proof of the connection for unweighted kernels given in \cite{Suszczynski} in the case when $W$ is bounded.

We shall begin with the definitions and basic facts used in this paper. Additionally, because we are dealing with the weighted Bergman kernels, we will recall for which weights in general the weighted Bergman kernel exists (although we are working here with differentiable weights only).

\section{Definitions and Notation}

   Let $W\subset\CC$ be a domain, and let ${\cal W}(W)$ be the set
   of weights on $W$, i.e., ${\cal W}(W)$ is the set of all Lebesque measurable, real-valued,
   positive functions $\mu$ on $W$ (we consider two weights as equivalent if they are equal almost
   everywhere with respect to the Lebesque measure on $W$).
   For $\mu\in {\cal W}(W)$ we denote by $L^2(W,\mu)$ the space of all Lebesque measurable, complex-valued,
   $\mu$-square integrable functions on $W$, equipped with the norm $||\cdot||_{W, \mu}:=||\cdot||_{\mu}$
   given by the scalar product
   \begin{align*}
   <f|g>_{\mu}:=\int_W f(z)\overline{g(z)}\mu(z)dV,\quad f, g\in L^2(W,\mu).
   \end{align*}
   The space $L^2_H(W,\mu)=\hbox{\rm Hol}(W)\cap L^2(W,\mu)$ is called the \textit{weighted Bergman space}, where $\hbox{\rm Hol}(W)$ denotes the space of all holomorphic functions on the domain $W$. For any $z\in W$ we define the evaluation functional $E_z$ on $L^2_H(W,\mu)$ by the formula
   \begin{align*}
   E_zf:=f(z),\quad f\in L^2_H(W,\mu).
   \end{align*}
   Let us recall the definition [Def. 2.1] of admissible weight given in \cite{Pasternak-Winiarski1}.
   \begin{definition}[Admissible weight]
   A weight $\mu\in {\cal W}(W)$ is called an {\textit{admissible weight}}, an a-weight for short, if $L^2_H(W,\mu)$
   is a closed subspace of $L^2(W,\mu)$ and for any $z\in W$, the evaluation functional $E_z$ is
   continuous on $L^2_H(W,\mu)$. The set of all a-weights on $W$ will be denoted by ${\cal AW}(W)$.
   \end{definition}
   The definition of admissible weight provides us with existence and uniqueness of the related Bergman kernel and completeness of the space $L^2_H(W,\mu)$.
   The concept of a-weight was introduced in \cite{Pasternak-Winiarski}, and in \cite{Pasternak-Winiarski1} several theorems concerning admissible weights are proved. An illustrative result is:
   \begin{theorem}\cite[Cor. 3.1]{Pasternak-Winiarski1}
   Let $\mu\in {\cal W}(W)$. If the function $\mu^{-a}$ is locally integrable on $W$ for some $a>0$ then $\mu\in {\cal AW}(W)$.
   \end{theorem}
   Now, let us fix a point $t\in W$ and minimize the norm $||f||_{\mu}$ in the class $E_t=\{f\in L^2_H(W,\mu); f(t)=1\}$.
   It can be proved, in a fashion similar to the classical case, that if $\mu$ is an admissible weight then there exists exactly one function minimizing the norm. Let us denote it by $\phi_{\mu}(z,t)$.
   The \textit{weighted Bergman kernel function $K_{W,\,\mu}$} is defined as follows :
   \begin{align*}
   K_{W,\,\mu}(z,t)=\frac{\phi_{\mu}(z,t)}{||\phi_{\mu}||_{\mu}^2}.
   \end{align*}

    \section{From the Unweighted to the Weighted Case}

    It is well known that a Green's function for the Laplace operator takes the form $$G_W(z,w)=h_W(z,w)-\ln|z-w|,$$ where $h_W$ is harmonic w.r.t $z\in W.$
    Thus
$$
\frac{\partial^2 G_W}{\partial z \partial\overline{w}}=\frac{\partial}{\partial z}\left(\frac{\partial h_W}{\partial\overline{w}}- \frac{1}{2} \frac{\partial}{\partial\overline{w}}
\bigl [ \ln (z - w) + \ln(\overline{z} - \overline{w}) \bigr ] \right) =
    \frac{\partial^2 h_W}{\partial z \partial\overline{w}}
$$
Moreover $h_W(z,w)=G_W(z,w)+\ln|z-w|=G_W(w,z)+\ln|w-z|
    =h_W(w,z)$. Thus $h_W$ is harmonic with respect to $z$ and $w.$
It turns out that (similarly as in the classical case (see \cite{Suszczynski})) regularity of $\partial W$
is not important at all.
\begin{theorem}
If $\rho(z)=|\mu(z)|^2$, where $\mu\in \hbox{\rm Hol}(\overline{W})$, and has no zeros on $\overline{W}$, then
$$
K_{W,\,\rho}(z,w)=-\frac{2}{\pi\rho(z)\rho(w)}\frac{\partial^2}{\partial z \partial \overline{w}}G_{W,\,\rho}(z,w).
$$
\end{theorem}

\begin{proof}
    It is well known that any domain $W\subset\CC$ may be written as
$$
W=\bigcup_{j=1}^{\infty}W_j,\quad W_1\Subset W_2\Subset W_3\Subset\ldots \, ,
$$
    where $\partial W_j$ consists of a finite number of Jordan curves (we do not assume any regularity of $\partial W$), for any $j\in\NN$. Let $\rho_j(z)=|\mu_j(z)|^2$ where $\mu_j\in \hbox{\rm Hol}(\overline{W_j})$ are such that    $\mu_j(z)\xrightarrow[j\to\infty]{}\mu(z)$ pointwise on $W$ and $|\mu_j(z)|^2\leq |\mu_{j+1}(z)|^2\leq|\mu(z)|^2$ on $W_j$. We denote by $G_{W_j}(z,w)$ the Green's function of $W_j$ and $K_{W_j,\,\rho_j}(z,w)$ the weighted Bergman kernel of $W_j,\,\,j\in\NN$. Under these assumptions,
    $K_{W_j,\,\rho_j}(z,w)\rightarrow K_{W,\,\rho}$ locally uniformly on $W\times W$ (see \cite{Wojcicki}).
    The proof of the classical equality for the unweighted Bergman kernel and the Green's function of the Laplace operator given in \cite{Suszczynski} is based on two steps :

    \begin{enumerate}
      \item[Step 1] : $\displaystyle (h_{W_j})_{j=1}^{\infty}$ is convergent to $h_W$ in $W\times W$
      \item[Step 2] : $\displaystyle \left (\frac{\partial^2 h_{W_j}}{\partial z \partial\overline{w}} \right )_{j=1}^{\infty}$ is convergent to $\displaystyle \frac{\partial^2 h_W}{\partial z \partial\overline{w}}$ in $W\times W$.
    \end{enumerate}

    The proof of Step $1$ (given in \cite{Suszczynski}) is based on the fact, that a sequence $(h_{W_j})_{j=1}^{\infty}$ is increasing ( since $h_{W_j}(z,w)=G_{W_j}(z,w)+\ln|z-w|$ ) and bounded from above by $h_W(z,w)$. Thus, by the Harnack theorem, it is convergent to a harmonic function $\widetilde{h}$.
    It is shown that $h_W=\widetilde{h}$, using the fact that $0$ is the greatest harmonic minorant of $G_{W_j}$ and $h_{W_j}$ is the least harmonic majorant of $\ln|z-w|$ on $W$ (see \cite{Helms}). The proof of step $2$ is based on the Poisson formula:
    $$
    \begin{array}{lll}
    \displaystyle h_{W_j}(z,w)= \\[12pt]
    \displaystyle=\frac{1}{(2\pi)^2}\int\int_{[0,2\pi]^2}h_{W_j}(z_0+re^{it},w_0+re^{is})
    \RE\left(\frac{re^{it}+z}{re^{it}-z}\right)\RE\left(\frac{re^{is}+w}{re^{is}-w}\right)dtds \\
    \end{array}
    $$ for $(z_0,w_0)\in W\times W$ and $(z,w)$ in a neighborhood of $(z_0,w_0)$. Taking $\displaystyle\frac{\partial^2}{\partial z\partial\overline{w}}$ on both sides we get in the limit with $j\to\infty$ that $$\lim_{j\to\infty}\frac{\partial^2h_{W_j}(z,w)}{\partial z\partial\overline{w}}=
    \frac{\partial^2h_W(z,w)}{\partial z\partial\overline{w}}=\frac{\partial^2G_W(z,w)}{\partial z\partial\overline{w}}$$

    We can simplify the proof of each of these assertions by using Harnack's theorem on harmonic functions.
    Taking $\lim_{j\to\infty}$ in
    $$
    G_{W_j}(z,w)=h_{W_j}(z,w)-\ln|z-w| \, ,  \eqno (3.1)
    $$
    we get
    $$
h_W(z,w)-\ln|z-w|=G_W(z,w)=\lim_{j\to\infty}h_{W_j}(z,w)-\ln|z-w| \, ,
$$
thus $\displaystyle h_{W_j}\xrightarrow[j\to\infty]{} h_W.$ Moreover, since $W_j\subset W_{j+1}$, we have
    $$
h_{W_j}(z,w)=G_{W_j}(z,w)+\ln|z-w|\leq G_{W_{j+1}}(z,w)+\ln|z-w|=h_{W_{j+1}}(z,w)\, ,
$$ for $z,\,w\in W.$
    Similarly $h_{W_j}(z,w)\leq h_W(z,w),$ for $z,\,w\in W.$ A nondecreasing sequence of harmonic functions
$$
h_{W_1}\leq h_{W_2}\leq h_{W_3}\leq\ldots\leq h_W
$$
is convergent locally uniformly in $W$ to a (harmonic function) $h_W$ by Harnack's theorem, thus (again by Harnack's theorem) $h_{W_j}\to h_W$ in $C^{\infty}(W).$ So
$$
\lim_{j\to\infty}
    \frac{\partial^2 h_{W_j}}{\partial z \partial\overline{w}}=\frac{\partial^2 h_W}{\partial z \partial\overline{w}} \,.
$$
    One can find in (\cite{Garabedian}, p.494) that $G_{W_j,\,\rho_j}(z,w)=\overline{\mu_j(z)}\mu_j(w)G_{W_j}(z,w)$.
    We conclude that
    \begin{lemma}
    If $L^2_H(W_j)\neq \{0\},$ then $G_{W_j,\,\rho_j}$ exists (and is given by (3.1)).
    \end{lemma}
    Now let us observe that
    \[
     \begin{array}{lll}
      \displaystyle \frac{\partial^2 G_{W_j,\,\rho_j}(z,w)}{\partial z\partial\overline{w}}&=&\displaystyle
      \frac{\partial^2}{\partial z\partial\overline{w}}\left(\overline{\mu_j(z)}\mu_j(w)G_{W_j}(z,w)\right)\\[12pt]
      &=&\displaystyle\frac{\partial}{\partial z}\left(\overline{\mu_j(z)} \left (\frac{\partial\mu_j(w)}{\partial\overline{w}}
      G_{W_j}(z,w)+\mu_j(w)\frac{\partial G_{W_j}(z,w)}{\partial\overline{w}}\right )\right) \\[12pt]
      &=&\displaystyle \frac{\partial}{\partial z}\left (\overline{\mu_j(z)}\mu_j(w)\frac{\partial G_{W_j}(z,w)}{\partial\overline{w}} \right ) \\[12pt]
      &=& \displaystyle\overline{\mu_j(z)}\mu_j(w)\frac{\partial^2 G_{W_j}(z,w)}{\partial z\partial\overline{w}}.\\
     \end{array}
   \]
    By the regularity of any $\partial W_j$ we have
$$
K_{W_j,\,\rho_j}(z,w)=-\frac{2}{\pi\rho_j(z)\rho_j(w)}\frac{\partial^2}{\partial z \partial \overline{w}}G_{W_j,\,\rho_j}(z,w) \, ,
$$ which in the limit as $j\to\infty$ yields

    \[
     \begin{array}{lll}
      \displaystyle K_{W,\,\rho}(z,w)&=&\displaystyle-\frac{2}{\pi\rho(z)\rho(w)}\overline{\mu(z)}\mu(w)
      \frac{\partial^2}{\partial z \partial \overline{w}}G_{W}(z,w) \\[12pt]
      &=&\displaystyle -\frac{2}{\pi\rho(z)\rho(w)}\frac{\partial^2 G_{W,\,\rho}(z,w)}{\partial z\partial\overline{w}}.\\
     \end{array}
   \]
\end{proof}

\subsection{\bf Non-Holomorphic Weights}
On closer scrutiny, the crucial thing in the proof of Theorem $3.1$ was to relate the weighted Green function to the unweighted one. This relationship turns out to be preserved even if we relax the assumption about holomorphicity of weight, as the following reveals:
\begin{theorem}
If $\rho(z)=|\mu(z)|^2$, where $\mu$ is a continuously differentiable function
of $x$ and $y$ on a neighborhood of $\overline{W}$ (and has no zeros on $\overline{W}$) then
$$
K_{W,\,\rho}(z,w)=-\frac{2}{\pi\rho(z)\rho(w)}\frac{\partial^2}{\partial z \partial \overline{w}}G_{W,\,\rho}(z,w).
$$
\end{theorem}

\begin{proof}
Let $\{W_j\}_{j=1}^{\infty}$ and $\rho_j$ be as in the proof of Theorem $3.1$
The crucial thing is to find $g_j(z)$ such that $u_j(w)=g_j(w)U_j(w)$ is a general solution of the equation
$$
\frac{\partial}{\partial\overline{w}}\frac{1}{\rho_j(w)}\frac{\partial}{\partial w}u_j(w)=0 \, ,
$$
and $U_j(w)$ is (an arbitrary) complex and harmonic on $W_j$ (we define $g$ on the same way by means of $P_{\rho}$).
Thus
\[
     \begin{array}{lll}
       0&=&\displaystyle\frac{\partial}{\partial\overline{w}}\frac{1}{\rho_j(w)}\frac{\partial}{\partial w}u_j(w)
      = \displaystyle\frac{\partial}{\partial\overline{w}}\frac{1}{\mu_j\overline{\mu_j}}
      \frac{\partial}{\partial w}(g_j(w)U_j(w)) \\[12pt]
      &=&\displaystyle\frac{\partial}{\partial\overline{w}}(\frac{1}{\mu_j\overline{\mu_j}}\frac{\partial g_j}{\partial w}U_j) + \displaystyle\frac{\partial}{\partial\overline{w}}(\frac{1}{\mu_j\overline{\mu_j}}g_j\frac{\partial U_j}{\partial w})\\[12pt]
     &=&\displaystyle (\frac{\partial}{\partial\overline{w}}\frac{1}{\mu_j\overline{\mu_j}})\frac{\partial g_j}{\partial w}U_j +\frac{1}{\mu\overline{\mu}}\left(\frac{\partial^2g_j}{\partial\overline{w}\partial w}U_j+\frac{\partial g_j}{\partial w}\frac{\partial U_j}{\partial\overline{w}}\right) +(\frac{\partial}{\partial\overline{w}}\frac{1}{\mu_j\overline{\mu_j}})g_j\frac{\partial U_j}{\partial w} \\[12pt]
     &+&\displaystyle\frac{1}{\mu_j\overline{\mu_j}}\left(\frac{\partial g_j}{\partial\overline{w}}\frac{\partial U_j}{\partial w}+g_j\underbrace{\frac{\partial^2U_j}{\partial\overline{w}\partial w}}_{0}\right)
  \end{array}
   \]
Thus
$$
     \left\{ \begin{array}{rll}
     \displaystyle\left(\frac{\partial}{\partial\overline{w}}\frac{1}{\mu_j\overline{\mu_j}}\right)
     \frac{\partial g_j}{\partial w}+
     \frac{1}{\mu_j\overline{\mu_j}}\frac{\partial^2g_j}{\partial\overline{w}\partial w}&=&0 \\[12pt]
     \displaystyle\frac{\partial g_j}{\partial w}&=&0 \\[12pt]
      \displaystyle \left(\frac{\partial}{\partial\overline{w}}\frac{1}{\mu_j\overline{\mu_j}}\right)g_j+
     \frac{1}{\mu_j\overline{\mu_j}}\frac{\partial g_j}{\partial\overline{w}}&=&0 \\
     \end{array} \right.
$$
\begin{remark}
By the equation above, $g_j$ is an antiholomorphic function.
\end{remark}
Examining the system above, we see that the first equation is a consequence of the second one. Let us focus on the third one:
$$
\left(\frac{\partial}{\partial\overline{w}}\frac{1}{\mu_j\overline{\mu_j}}\right)g_j+
     \frac{1}{\mu_j\overline{\mu_j}}\frac{\partial g_j}{\partial\overline{w}}=0 \, .
$$
It may be written in the form
$$
\frac{1}{g_j}\frac{\partial g_j}{\partial\overline{w}}=\frac{1}{\mu_j\overline{\mu_j}}\frac{\partial}{\partial\overline{w}}
(\mu_j\overline{\mu_j})
$$

Thus, for a given $\mu_j$, there is a function $g_j$ which must satisfy :

$$
     \left\{ \begin{array}{rll}
     \displaystyle \frac{1}{g_j}\frac{\partial g_j}{\partial\overline{w}}&=&\displaystyle
     \frac{1}{\mu_j\overline{\mu_j}}\frac{\partial}{\partial\overline{w}}(\mu_j\overline{\mu_j})\\[12pt]
     \displaystyle\frac{\partial g_j}{\partial w}&=&0 \\
     \end{array} \right.
$$
Notice that, if $\mu_j$ is holomorphic and $g_j=\overline{\mu_j}$, then the system above is satisfied (in this case we get the result of \cite{Garabedian}).
We may proceed to get the exact form of $g_j(z)$, namely :
$$
     \begin{array}{rll}
      \displaystyle\frac{\partial}{\partial\overline{w}}\ln g_j&=&\displaystyle\frac{\partial}{\partial\overline{w}}\ln(\mu_j\overline{\mu_j}) \\[12pt]
      \displaystyle\ln g_j&=&\displaystyle\ln(\mu_j\overline{\mu_j}) + h_j(w) \\[12pt]
      \displaystyle g_j(w)&=&\displaystyle \mu_j\overline{\mu_j}e^{h_j(w)}=|\mu_j(w)|^2e^{h_j(w)} \\
     \end{array}
$$
where $h_j\in C^1(\overline{W_j})$.
But $g_j$ is antiholomorphic, so
$$
\begin{array}{rll}
      0&=&\displaystyle\frac{\partial g_j}{\partial w}=\frac{\partial}{\partial w}(\mu_j\overline{\mu_j})e^h_j
      +\mu_j\overline{\mu_j}e^h\frac{\partial h_j}{\partial w}
       \\[12pt]
      0&=& \displaystyle\frac{\partial}{\partial w}(\mu_j\overline{\mu_j})+\mu_j\overline{\mu_j}\frac{\partial h_j}{\partial w} \\
\end{array}
$$
So to sum up
$$
g_j(z)=|\mu_j(z)|^2e^{h_j(z)} \, ,
$$ where
$$
\frac{\partial}{\partial w}(\mu_j\overline{\mu_j})+\mu_j\overline{\mu_j}\frac{\partial h_j}{\partial w} = 0 \, .
$$
In fact, we may choose $h_j$ in such a way that $h_j\rightarrow h$, where $h\in C^1(\overline{W})$
(because such a $g_j$ will still be a good solution to the posed problem). Then $g_j\rightarrow g=|\mu(z)|^2e^{h(z)}$.

Again (as in \cite{Garabedian})
$$
G_{W_j,\,\rho_j}(z,w)=g_j(z)\overline{g_j(w)}G_{W_j}(z,w) \, ,
$$
so as previously
$$
\frac{\partial^2 G_{W_j,\,\rho_j}(z,w)}{\partial z\partial\overline{w}}=g_j(z)\overline{g_j(w)}
     \frac{\partial^2 G_{W_j}(z,w)}{\partial z\partial\overline{w}}
$$
and
$$
K_{W,\,\rho}(z,w)=-\frac{2}{\pi\rho(z)\rho(w)}\frac{\partial^2 G_{W,\,\rho}(z,w)}{\partial z\partial\overline{w}}\, .
$$
\end{proof}

\section{Concluding Remarks}

It is well established that weighted Bergman spaces are both
intrinsically interesting and a powerful analytic tool. Our
purpose in this paper has been to develop this set of ideas,
and particularly the connection between the Bergman kernel
and the Green's function in the weighted context. Some of the applications
might be :
\begin{itemize}
  \item[a)] With the established connection between weighted Bergman kernel and Green's function in hand, we can reformulate the weighted version of the so called ``small conjecture'' (it is the so called Skwarczy\'nski distance equivalent to the Bergman distance, see \cite{Skwarczynski2, Skwarczynski1} ) as :
  \begin{remark}
  Assume $W\Subset\CC$, and $\mu$ is a continuously differentiable function of $x$ and $y$ on a neighborhood of $\overline{W}$. Then $t_n\to t\in\partial W$ represents defective evaluation iff
  $$-\frac{2}{\pi\rho(z)\rho(w)}\frac{\partial^2}{\partial z \partial \overline{w}}G_{W,\,\rho}(z,w)(\cdot,t_n)\to
  \gamma$$ weakly in $L^2_H(W,\mu)$ and
  $$-\frac{2}{\pi\rho(z)\rho(w)}\frac{\partial^2}{\partial z \partial \overline{w}}G_{W,\,\rho}(z,w)(t_n,t_n)\to
  \kappa^2$$ where $||\gamma||\neq\kappa$.
  This is important, since the involved so called Skwarczy\'nski distance is biholomorphically invariant, and given more explicitly than the Bergman distance.

  \end{remark}

  \item[b)] Using the method of alternating projections (see \cite{Skwarczynski}), we can recover (having some Dirichlet and Neuman boundary conditions on $G_W$) $G_{W,\mu}$ for an arbitrary domain $W$ lying in $\CC$.
\end{itemize}

\bigskip

{\noindent
Steven G. Krantz\\
Department of Mathematics\\
Washington University in St. Louis\\
St. Louis, Missouri, 63130 USA\\
E-mail: \,sk@math.wustl.edu}

\medskip

{\noindent
Pawe{\l} M. W\'ojcicki\\
Faculty of Mathematics and Information Science\\
Warsaw University of Technology\\
Koszykowa 75,  00-662 Warsaw, Poland\\
E-mail: \,p.wojcicki@mini.pw.edu.pl}

\bigskip



\rightline{}

\end{document}